\documentclass[12pt]{amsart}

\usepackage{amssymb}
\usepackage{amscd}
\usepackage{statex}

\title[A birational embedding with two Galois points for certain curves]{A birational embedding with two Galois points for certain Artin--Schreier curves}
\author{Satoru Fukasawa \and Kazuki Higashine}

\subjclass[2010]{14H50, 14H37}
\keywords{Galois point, Galois group, plane curve, automorphism group}
\thanks{The first author was partially supported by JSPS KAKENHI Grant Number 16K05088}
\address[Satoru Fukasawa]{Department of Mathematical Sciences, Faculty of Science, Yamagata University, Kojirakawa-machi 1-4-12, Yamagata 990-8560, Japan}
\email{s.fukasawa@sci.kj.yamagata-u.ac.jp}
\address[Kazuki Higashine]{Graduate School of Science and Engineering, Yamagata University, Kojirakawa-machi 1-4-12, Yamagata 990-8560, Japan}
\email{s16m111m@st.yamagata-u.ac.jp}

\newtheorem{thm}{Theorem}
\newtheorem{prop}{Proposition}

\newtheorem{lem}{Lemma}

\theoremstyle{definition}
\newtheorem{rem}{Remark}

\newcommand{\pline}[2]{{\overline{#1#2}}}

\newcommand{\map}[3]{{#1: #2\to #3}}

\newcommand{\ratmapele}[5]{{#1: #2\dashrightarrow #3\ ;\ #4 \mapsto #5}}
\newcommand{\image}[2]{{#1 (#2)}}

\newcommand{\dval}[2]{{{\rm ord}_{#1}\left( #2\right)}}

\newcommand{\proj}[1]{{\pi_{#1}}}
\newcommand{\rproj}[1]{{{\hat{\pi}}_{#1}}}
\newcommand{\pproj}[1]{{\pi_{#1}^{\ast}}}

\begin{document}
\begin{abstract}
We show that two curves of Artin--Schreier type have a birational embedding into a projective plane with two Galois points. As a consequence, all curves with large automorphism groups in the classification list by Henn have a birational embedding with two Galois points.  
\end{abstract}
\maketitle

\section{Introduction}
Let $K$ be an algebraically closed field of characteristic $p>0$, and let $C\subset \mathbb{P}^{2}$ be an irreducible plane curve of degree $d\geq 3$ over $K$. Take a point $P\in \mathbb{P}^{2}$. We consider the projection
\[ \ratmapele{\proj{P}}{C}{\mathbb{P}^{1}}{Q}{\pline{P}{Q}}\]
with the center $P$, where $\pline{P}{Q}$ represents the line passing through $P$ and $Q$ if $P\not= Q$. Since $\proj{P}$ is a dominant rational map, $\proj{P}$ induces a field extension $K(C)/\pproj{P}K(\mathbb{P}^{1})$ of function fields. If the extension $K(C)/\pproj{P}K(\mathbb{P}^{1})$ is Galois, then $P$ is called a Galois point for $C$. This notion was introduced by H. Yoshihara (\cite{miura-yoshihara, yoshihara}). Furthermore, a Galois point $P$ is said to be inner (resp. outer) if $P$ is a smooth point of $C$ (resp. a point not contained in $C$).

Before the paper \cite{birationalembedding} appeared, there were only seven types of examples of plane curves with two inner Galois points. In \cite{birationalembedding}, a criterion for the existence of a birational embedding of smooth projective curves with two Galois points was presented. Using this criterion, the first author and Waki have obtained (at least) eight new examples of plane curves (whose smooth models are rational, elliptic, Hermitian, Suzuki or Ree curves) with two inner Galois points (see the Table and references in \cite{listofproblem}). In this article, we present further new examples of plane curves with two Galois points.

 Let $(X:Y:Z)$ be a  system of homogeneous coordinates of $\mathbb{P}^{2}$. An affine open set defined by $Z\not= 0$ is denoted by $U_{Z}$, and $(x, y)=(X/Z, Y/Z)$ is a system of affine coordinates of $U_{Z}$. Put $q=p^{n}$, where $n$ is a positive integer. We consider the following two curves:
 \begin{itemize}
\item[(1)] The curve $\mathcal{F}_{m}\subset \mathbb{P}^{2}$ is (the projective closure of the affine curve) defined by
\[ y^{m}=x^{q}+x,\]
where $m$ is an integer dividing $q+1$ with $2\leq m< q$.
\item[(2)] The curve $\mathcal{G}_{r}\subset \mathbb{P}^{2}$ is (the projective closure of the affine curve) defined by
\[ y^{q^{r}+1}=x^{q}+x,\]
where $r$ is an integer with $r\geq 2$.
\end{itemize}

\begin{rem}\label{remark1}
The curve $\mathcal{F}_{2}$ is the same as the curve in (I\hspace{-.1em}I) of \cite[Theorem 11.127]{hkt}. Similarly, the curve $\mathcal{G}_{r}$ with $q=2$ is the same as the curve in (I) of \cite[Theorem 11.127]{hkt}.
\end{rem}

The following two theorems are our main results.

\begin{thm}\label{theorem1}
Let $\hat{\mathcal{F}}_{m}$ be the smooth model of $\mathcal{F}_{m}$. Then the following hold.
\begin{itemize}
\item[(1)] There exists a morphism $\map{\varphi}{\hat{\mathcal{F}}_{m}}{\mathbb{P}^{2}}$ such that $\varphi$ is birational onto its image, ${\rm deg}(\image{\varphi}{\hat{\mathcal{F}}_{m}})=q+1$, and $\image{\varphi}{\hat{\mathcal{F}}_{m}}$ has two distinct inner Galois points.
\item[(2)] There exists a morphism $\map{\psi}{\hat{\mathcal{F}}_{m}}{\mathbb{P}^{2}}$ such that $\psi$ is birational onto its image, ${\rm deg}(\image{\psi}{\hat{\mathcal{F}}_{m}})=q+1$, and $\image{\psi}{\hat{\mathcal{F}}_{m}}$ has two distinct outer Galois points.
\end{itemize}
\end{thm}

\begin{thm}\label{theorem2}
Let $\hat{\mathcal{G}}_{r}$ be the smooth model of $\mathcal{G}_{r}$. Then there exists a morphism $\map{\xi}{\hat{\mathcal{G}}_{r}}{\mathbb{P}^{2}}$ such that $\xi$ is birational onto its image, ${\rm deg}(\image{\xi}{\hat{\mathcal{G}}_{r}})=q^{r}+1$, and  $\image{\xi}{\hat{\mathcal{G}}_{r}}$ has two distinct outer Galois points.
\end{thm}

For algebraic curves $C$ with genus $g\geq 2$, there exist only four types of curves such that $| {\rm Aut}_{K}(K(C)) | \geq 8g^{3}$ by Henn (\cite[Theorem 11.127]{hkt}). Combining with results of Homma \cite{homma} and of the first author \cite{hermitiananddlscurve} for the Hermitian and Suzuki curves, we infer that all curves in the list \cite[Theorem 11.127]{hkt} have a birational embedding with two Galois points.

\section{Proof of main theorems}
For the proof, we use \cite[Theorem 1 and Remark 1(2)]{birationalembedding}. Note that $\mathcal{F}_{m}$ (resp. $\mathcal{G}_{r}$) has a unique singular point $(0:1:0)$ (resp. $(1:0:0)$), if $(q, m)\not= (3, 2)$. The fiber of the normalization over $(0:1:0)$ (resp. over $(1:0:0)$) consists of a unique point $P_{\infty}\in \hat{\mathcal{F}}_{m}$ (resp. $Q_{\infty} \in \hat{\mathcal{G}}_{r}$). $\rproj{P}$ denotes the composition of $\proj{P}$ and the normalization. The ramification index of $\rproj{P}$ at a ramification point $Q$ over a branch point $R$ is represented by $e(Q|R)$.

\begin{proof}[\bf{Proof of Theorem 1(1)}]
We consider the subgroup
\[ G_{1} : =\{ (X:Y:Z)\mapsto (X+\lambda :Y:Z)\mid \lambda \in K, \lambda^{q}+\lambda =0\} \subset {\rm PGL}(3, K).\]
For any $\sigma \in G_{1}$, $\sigma$ induces an automorphism of $\hat{\mathcal{F}}_{m}$. Put $s=(q+1)/m$. We consider the rational map
\[ \alpha : \mathcal{F}_{m}\dashrightarrow \mathbb{P}^{2}\ ;\ \left( \frac{1}{x}:\frac{y}{x^{s}}:1\right) .\]
Then $\alpha$ is a birational map of $\mathcal{F}_{m}$. In fact,
\[ \left( \frac{1}{x}\right)^{q}+\left( \frac{1}{x}\right) =\frac{x^{q}+x}{x^{q+1}}=\frac{y^{m}}{x^{sm}}.\]
Put $G_{2}=\alpha^{-1}G_{1}\alpha$. We define the set
\[ W : =\{ P_{\infty}\} \cup \{ P_{\lambda}=(\lambda :0:1)\mid \lambda \in K, \lambda^{q}+\lambda =0 \} .\]
It follows that $\sigma (W\setminus \{ P_{\infty}\})=W\setminus \{ P_{\infty}\}$ and $\sigma (P_{\infty})=P_{\infty}$, for any $\sigma \in G_{1}$. On the other hand, $\tau (W\setminus \{ P_{0}\})=W\setminus \{ P_{0}\}$ and $\tau (P_{0})=P_{0}$ hold for all $\tau \in G_{2}$. In fact, the automorphism $\alpha$ interchanges $P_{0}$ and $P_{\infty}$, and preserves the set $W\setminus \{ P_{0}, P_{\infty}\}$. Since the fixed point by $G_{1}$ is different from that by $G_{2}$, the group $G_{1}\cap G_{2}$ is trivial. Furthermore, it follows that $K(\hat{\mathcal{F}}_{m})^{G_{1}}=K(1/y)\simeq K(\mathbb{P}^{1})$ and $K(\hat{\mathcal{F}}_{m})^{G_{2}}=K(x^{s}/y) \simeq K(\mathbb{P}^{1})$, and the equality
\[ \{ P_{\infty}\} \cup \{ \sigma (P_{0})\mid \sigma \in G_{1}\} =W=\{ P_{0}\} \cup \{ \tau (P_{\infty})\mid \tau \in G_{2}\}\]
holds. Applying \cite[Theorem 1]{birationalembedding}, we have assertion (1).
\end{proof}

We can determine the number of inner Galois points on $\image{\varphi}{\hat{\mathcal{F}}_{m}}$. 

\begin{prop}
If $(q, m)\not= (3, 2)$, then there exist at most two inner Galois points on $\image{\varphi}{\hat{\mathcal{F}}_{m}}$.
\end{prop}

\begin{proof}
As in the proof of \cite[Theorem 1]{birationalembedding}, the birational map given above is represented by
\[ \varphi =\left( \frac{1}{y}:\frac{x^{s}}{y}:1\right) .\]
We consider the image $\image{\varphi}{W}$. It can be easily checked that $\image{\varphi}{W}=\image{\varphi}{\hat{\mathcal{F}}_{m}}\cap \{ Z=0\}$. More precisely, $\varphi (P_{\infty})=(0:1:0)$ and $\varphi (P_{\lambda})=(1:\lambda^{s}:0)$ for all $P_{\lambda}\in W \setminus \{ P_{\infty}\}$.

Let $\varphi (R)$ be an inner Galois point. Put $H={\rm Gal}(K(\mathcal{F}_{m})/\pproj{\varphi (R)}(K(\mathbb{P}^{1})))$. Since $\varphi (R)$ is inner Galois, ${\rm deg}(\proj{\varphi (R)})= | H | =q$. Since $| {\rm Aut}(\hat{\mathcal{F}}_{m}) | =m(q-1)q(q+1)$ (see \cite[Theorem 12.11]{hkt}, \cite{stichtenoth1973}), $H$ is a Sylow $p$-subgroup of ${\rm Aut}(\hat{\mathcal{F}}_{m})$. Thus, there exists a point $P\in W$ such that $\eta (P)=P$ for any $\eta \in H$. Using \cite[Corollary 3.8.2]{stichtenoth}, we know that $\rproj{\varphi (R)}$ is totally ramified at $P$.

We show that $\varphi (R)\in \{ Z=0\}$. Assume by contradiction that $\varphi (R)\not\in \{ Z=0\}$. The projection with the center $\varphi (P_{\infty})=(0:1:0)$ is totally ramified at $P_{\infty}$. In fact, $e(P_{\infty}|\rproj{\varphi (P_{\infty})}(P_{\infty}))=\dval{P_{\infty}}{1/y}=q$. Note that the intersection multiplicity of $\image{\varphi}{\hat{\mathcal{F}}_{m}}$ and the tangent line at $\varphi (P_{\infty})$ is equal to $q+1$. Since $\pline{\varphi (R)}{\varphi (P_{\infty})}$ is not a tangent line at $\varphi (P_{\infty})$, there exists a point $P'\not\in W$ such that $\varphi (P')\in \pline{\varphi (R)}{\varphi (P_{\infty})}$. On the other hand, $H$ acts on $\rproj{\varphi (R)}^{-1}(\pline{\varphi (R)}{\varphi (P_{\infty})})$ transitively by \cite[Theorem 3.7.1]{stichtenoth}. Thus, there exists $\eta '\in H$ which satisfies $\eta '(P_{\infty})=P'$. Note that $W$ is the set of all points of $\hat{\mathcal{F}}_{m}$ such that the Weierstrass semigroup at the point generated by $m$ and $q$ (see \cite[Theorem 12.10]{hkt}). The invariance of Weierstrass points under the action of ${\rm Aut}(\hat{\mathcal{F}}_{m})$ gives a contradiction. Therefore, $\varphi (R)\in \{ Z=0\}$.

We show that $R=P$. If $\varphi (R)\not= \varphi (P)$, then the fiber of $\rproj{\varphi (R)}$ over $\pline{\varphi (R)}{\varphi (P)}\in \mathbb{P}^{1}$ coincides with $W\setminus \{ R\}$. However, $\rproj{\varphi (R)}$ is totally ramified at $P$. This is a contradiction. Thus, $R=P$. 

To conclude $R=P_{0}$ or $P_{\infty}$, we consider the projection with the center $\varphi (Q)$, where $\varphi (Q)\in \varphi (W)\setminus \{ \varphi (P_{0}), \varphi (P_{\infty})\}$. Let $Q=(\lambda :0:1)$ with $\lambda^{q-1}+1=0$. The morphism $\rproj{\varphi (Q)}$ is represented by
\[ \rproj{\varphi (Q)}=\left( \frac{x^{s}-\lambda^{s}}{y}:1\right) .\]
Since $\dval{Q}{x^{s}-\lambda^{s}}=m$, we have
\[ e(Q|\rproj{\varphi (Q)}(Q))=\dval{Q}{\frac{x^{s}-\lambda^{s}}{y}}=m-1<q.\]
Therefore, we have $\varphi (R)=\varphi (P_{0})$ or $\varphi (P_{\infty})$, and get the assertion.
\end{proof}

Before the proof of assertion (2), we show the following lemma.

\begin{lem}\label{lemma1}
Under the same assumption for the curve $\mathcal{F}_{m}$, we consider the curve $\mathcal{E}_{m}: y^{m}-x^{q+1}+1=0$. Then $\mathcal{F}_{m}$ and $\mathcal{E}_{m}$ are birational to each other.
\end{lem}
\begin{proof}
First, we take $a\in K$ which satisfies $a^{q}+a-1=0$, and a coordinate change $(x, y)\mapsto (x-a, y)$. Then $\mathcal{F}_{m}$ is projectively equivalent to the curve $\={\mathcal{F}}_{m}$ defined by $y^{m}-x^{q}-x-1=0$. Next, we consider the rational map
\[ \={\mathcal{F}}_{m}\to \mathbb{P}^{2}\ ; \left( \frac{1}{x}: \frac{y}{x^{s}}: 1\right) ,\]
where $s=(q+1)/m$. Then, it follows that
\[ \left( \frac{1}{x}\right)^{q+1}+\left( \frac{1}{x}\right)^{q}+\left( \frac{1}{x}\right) =\frac{x^{q}+x+1}{x^{q+1}}=\left( \frac{y}{x^{s}}\right)^{m}.\]
Therefore, $\={\mathcal{F}}_{m}$ and the curve $y^{m}-x^{q+1}-x^{q}-x=0$ are birational to each other. Finally, we take a coordinate change $(x, y)\mapsto (x+1, y)$, and get the assertion.
\end{proof}

\begin{proof}[\bf{Proof of Theorem 1(2)}]
Let $\hat{\mathcal{E}}_{m}$ be the smooth model of $\mathcal{E}_{m}$. From Lemma $\ref{lemma1}$, we have only to show that $\hat{\mathcal{E}}_{m}$ has the morphism which satisfies assertion (2). Put $X=\{ (\zeta :0:1)\mid \zeta \in K, \zeta^{q+1}=1\}$. We consider the subgroup
\[ G_{1} : =\{ (X:Y:Z)\mapsto (\zeta X:Y:Z)\mid \zeta \in K, \zeta^{q+1}=1\} \subset {\rm PGL}(3, K).\]
For any $\sigma \in G_{1}$, $\sigma$ induces an automorphism of $\hat{\mathcal{E}}_{m}$. Further, it is easily checked that $G_{1}$ acts on the set $X$ transitively. 

We consider the rational map
\[ \beta : \mathcal{E}_{m}\dashrightarrow \mathbb{P}^{2}\ ;\ \left( \frac{x+\lambda (x-1)}{1+\lambda (x-1)}:\frac{y}{(1+\lambda (x-1))^{s}}:1\right) ,\]
where $\lambda \in K\setminus \{ 0, 1\}$ satisfies $\lambda^{q}+\lambda =0$, and $s=(q+1)/m$. Then $\beta$ acts on $\mathcal{E}_{m}$. In fact,
\begin{align*}
&\left( \frac{y}{(1+\lambda (x-1))^{s}}\right)^{m}-\left( \frac{x+\lambda (x-1)}{1+\lambda (x-1)} \right)^{q+1} +1\\
\\
&=\frac{y^{m}-(x+\lambda (x-1))^{q+1}+(1+\lambda (x-1))^{q+1}}{(1+\lambda (x-1))^{q+1}}
\end{align*}
and the numerator is equal to
\begin{align*}
x^{q+1}&-1-(x^{q+1}+\lambda (x-1) x^{q}+\lambda^{q}x (x-1)^{q}+\lambda^{q+1}(x-1)^{q+1})\\
+(1&+\lambda (x-1)+\lambda^{q}(x-1)^{q}+\lambda^{q+1}(x-1)^{q+1})\\
=-&\lambda (x-1) x^{q}+\lambda x (x^{q}-1)+\lambda (x-1)-\lambda (x^{q}-1)=0.
\end{align*}
It is easily checked that the inverse of $\beta$ is given by
\[ \beta^{-1}: \mathcal{E}_{m}\dashrightarrow \mathcal{E}_{m}\ ;\ \left( \frac{x-\lambda (x-1)}{1-\lambda (x-1)}:\frac{y}{(1-\lambda (x-1))^{s}}:1\right) .\]
Note that $\beta$ induces the bijection of the set $X$.

Put $G_{2}=\beta G_{1}\beta^{-1}$. We show that $G_{1}\cap G_{2}$ is trivial. Let $P=(0:\omega :1)\in \mathcal{E}_{m}$ with $\omega^{m}+1=0$. Then $\beta (P)=(\lambda /(\lambda -1): \omega /(1-\lambda)^{s}:1)$, and $\beta (P)$ is fixed by all elements of $G_{2}$. On the other hand, we take $\sigma =(\zeta x: y:1)\in G_{1}$ with $\zeta^{q+1}=1$ and $\zeta \not= 1$. Then
\[ \sigma (\beta (P))=\left( \frac{\zeta \lambda}{\lambda-1}: \frac{\omega}{(1-\lambda )^{s}}:1\right) \not= \left( \frac{\lambda}{\lambda-1}:\frac{\omega}{(1-\lambda)^{s}}:1\right) =\beta (P).\]
Thus, $G_{1}\cap G_{2}$ is trivial. Moreover, there are isomorphisms $K(\hat{\mathcal{E}}_{m})^{G_{1}}=K(y)\simeq K(\mathbb{P}^{1})$, $K(\hat{\mathcal{E}}_{m})^{G_{2}}=K(\beta^{*} (y))\simeq K(\mathbb{P}^{1})$, and the equality
\[ \{ \sigma (1:0:1)\mid \sigma \in G_{1}\} =X=\{ \tau (1:0:1)\mid \tau \in G_{2}\} \]
holds. Using \cite[Theorem 1, Remark 1(2)]{birationalembedding}, we have assertion (2).
\end{proof}

\begin{proof}[\bf{Proof of Theorem 2}]
We show Theorem $\ref{theorem2}$ in the same way as the proof of Theorem $\ref{theorem1}$(2). We consider the subgroup
\[ G_{1} : =\{ (X:Y:Z)\mapsto (X:\zeta Y:Z)\mid \zeta^{q^{r}+1}=1\} \subset {\rm PGL}(3, K).\]
For any $\sigma \in G_{1}$, $\sigma$ induces an automorphism of $\hat{\mathcal{G}}_{r}$. Further, it is not difficult to check that $Q_{\infty}$ is fixed by all elements of $G_{1}$.

Take $b, c\in K\setminus \{ 0\}$ such that $b=(-1)^{r-1}b^{q^{2r}}$ and $c^{q}+c=b^{q^{r}+1}$. We consider the polynomial
\[ g_{b}(y) : =\sum_{i=0}^{r-1}(-1)^{i}b^{q^{i+r}}y^{q^{i}}\in K[y] .\]
Then the polynomial $g_{b}(y)$ has the following properties:
\begin{itemize}
\item[(1)] $g_{-b}(y)=-g_{b}(y)$.
\item[(2)] $g_{b}(-a)=-g_{b}(a)$ for all $a\in K$.
\item[(3)] $g_{b}(y+a)=g_{b}(y)+g_{b}(a)$ for all $a\in K$.
\item[(4)] $(g_{b}(y))^{q}+g_{b}(y)=by^{q^{r}}+b^{q^{r}}y$.
\end{itemize}
We consider the rational map
\[ \gamma : \mathcal{G}_{r}\dashrightarrow \mathbb{P}^{2}\ ;\ (x:y:1)\mapsto (g_{b}(y)+c+x : y+b :1).\]
Then $\gamma$ is a birational map from $\mathcal{G}_{r}$ to itself. In fact,
\begin{align*}
(y+b)^{q^{r}+1}&=y^{q^{r}+1}+by^{q^{r}}+b^{q^{r}}y+b^{q^{r}+1}\\
&=(x^{q}+x)+((g_{b}(y))^{q}+g_{b}(y))+(c^{q}+c)\\
&=(g_{b}(y)+c+x)^{q}+(g_{b}(y)+c+x),
\end{align*}
and the inverse of $\gamma$ is given by
\[ \gamma^{-1}: \mathcal{G}_{r}\dashrightarrow \mathcal{G}_{r}\ ;\ (x:y:1)\mapsto (-g_{b}(y)+c'+x : y-b :1),\]
where $c'=-c+g_{b}(b)$. Note that $\gamma$ fixes $Q_{\infty}$.

Put $G_{2}=\gamma G_{1}\gamma^{-1}$. We show that $G_{1}\cap G_{2}$ is trivial. Let $R=(\alpha : 0:1)$ with $\alpha^{q}+\alpha =0$. Then $\gamma (R)=(c+\alpha : b :1)$, and $\gamma (R)$ is fixed by all elements of $G_{2}$. We consider $\sigma =(x: \zeta y:1)\in G_{1}$, where $\zeta \in K\setminus \{ 1\}$ and $\zeta^{q^{r}+1}=1$. Then, it follows that
\[ \sigma (\gamma (R))=(c+\alpha : \zeta b:1)\not= (c+\alpha : b :1)=\gamma (R) .\]
Thus, $G_{1}\cap G_{2}$ is trivial. Furthermore, $K(\hat{\mathcal{G}}_{r})^{G_{i}}\simeq K(\mathbb{P}^{1})$ for $i=1, 2$, and the equality
\[ \sum_{\sigma \in G_{1}}\sigma (Q_{\infty})=(q^{r}+1) Q_{\infty}=\sum_{\tau \in G_{2}}\tau (Q_{\infty})\]
holds as divisors. Applying \cite[Theorem 1 and Remark 1(2)]{birationalembedding}, we get Theorem 2.
\end{proof}

\end{document}